\documentclass{scrartcl}
\KOMAoptions{paper=a4, DIV=calc}

\usepackage{stix2}
\usepackage{microtype}

\usepackage[hang,flushmargin]{footmisc}

\usepackage[font=small,labelfont=bf]{caption}
\usepackage{xcolor}

\usepackage[
  backend=biber,
  style=alphabetic,
  sorting=ynt
]{biblatex}
\usepackage{xurl}
\usepackage{hyperref}

\AtEveryBibitem{\sloppy}

\usepackage{amsthm}
\usepackage{mathtools}

\usepackage{tabularx}
\usepackage{multirow}

\usepackage{algorithm}
\usepackage{algpseudocode}
\algrenewcommand\alglinenumber[1]{}

\usepackage{listings}

\lstdefinestyle{python}{
  language=Python,
  basicstyle=\ttfamily\small,
  breaklines=true,
  frame=single,
  tabsize=4,
}

\setkomafont{disposition}{\normalfont\bfseries}
\setkomafont{descriptionlabel}{\normalfont\bfseries}

\DeclareMathOperator{\im}{im}
\DeclareMathOperator{\hfk}{HFK}
\DeclareMathOperator{\gc}{GC}
\DeclareMathOperator{\gh}{GH}

\newcommand{\G}{\mathbb{G}}
\newcommand{\X}{\mathbb{X}}
\newcommand{\Oh}{\mathbb{O}}
\newcommand{\Z}{\mathbb{Z}}

\newcommand{\col}{\mathscr{c}}
\newcommand{\row}{\mathscr{r}}

\newcommand{\Rect}{\operatorname{Rect}(x,y)}
\newcommand{\eRect}{\operatorname{Rect}^\circ(x,y)}

\newcommand{\sbc}{\widehat{\gc}(\G)}
\newcommand{\sbh}{\widehat{\gh}(\G)}

\newcommand{\sbhds}[2]{\widehat{\gh}_{#1}(\G, #2)}
\newcommand{\sbdifferential}{\hat{\partial}_{\X, O_n}}

\newtheoremstyle{mystyle}
  {3pt}{3pt}
  {}
  {}
  {\bfseries}
  {}
  {.5em}
  {\thmname{#1}\thmnumber{ #2.}\thmnote{ (#3)}}

\theoremstyle{mystyle}
\newtheorem{theorem}{Theorem}[section]
\newtheorem{proposition}[theorem]{Proposition}
\newtheorem*{proposition*}{Proposition}
\newtheorem{lemma}[theorem]{Lemma}

\newtheorem{definition}[theorem]{Definition}
\newtheorem{example}[theorem]{Example}
\newtheorem{question}[theorem]{Question}
\newtheorem*{question*}{Question}

\title{Grid Diagrams of Fibered Knots}
\author{Paul Leon Itzlinger}
\date{}

\begin{document}

\maketitle

\addsec{Abstract}

Grid diagrams are special representations of knots in the three-sphere that are used to define a combinatorial version of knot Floer homology. Paolo Ghiggini and Yi Ni showed that knot Floer homology detects fibered knots \cite{ghiggini2008knot,ni2007knot}. Their results imply, in particular, that grid diagrams with a unique grid state whose Alexander grading is maximal only exist for fibered knots. Whether every fibered knot admits such a diagram remains an open question \cite[Open Problem 17.1.5]{ozsvath2015grid}.

Here, we investigate the existence of such special grid diagrams for fibered knots. We develop an efficient method for deciding whether a given grid diagram meets the even stricter condition of having a unique grid state that realizes an upper bound for the Alexander function. By implementing this method in a Python package, we find suitable grid diagrams for 5385 of the 5397 fibered prime knots with crossing number $\leq 13$.

\addsec{Acknowledgements}

I am very grateful to my master's thesis advisor, Prof. Vera Vértesi, who introduced me to this subject and whose continuous support was invaluable throughout the project.

\newpage

\addsec{Introduction}

Knot Floer homology is an invariant for links in three-manifolds, discovered by Ozsváth, Szabó, and Rasmussen \cite{ozsvath2004holomorphic_1, rasmussen2003floer}. The subject is based on Heegaard-Floer homology, an invariant for closed, orientable three-manifolds that uses tools from symplectic geometry. Knot Floer homology relies on counting pseudo-holomorphic disks. Sarkar and Wang discovered a combinatorial approach to counting these disks in special cases \cite{sarkar2010algorithm}. Based on these ideas, Manolescu, Ozsváth, Sarkar, and Thurston developed a combinatorial description of knot Floer homology for knots in the three-sphere $S^3$ \cite{manolescu2007combinatorial, manolescu2009combinatorial} out of which an independent, purely combinatorial homology theory for knots in $S^3$, called grid homology, was created \cite{ozsvath2015grid}.

For a knot $K$ in $S^3$, there exists a bigraded isomorphism between simply blocked grid homology $\widehat{\gh}(K)$ and knot Floer homology $\widehat{\hfk}(K)$ \cite{manolescu2009combinatorial}. Knot Floer homology is a powerful tool for investigating topological properties of knots, as it can be used to compute the Seifert genus and detect if a knot is fibered. Indeed, Ozsváth and Szabó showed that the Seifert genus is equal to the maximal integer $s$ for which $\widehat{\hfk}_*(K, s) = \bigoplus_d \widehat{\hfk}_d(K, s)$ is non-zero \cite{ozsvath2004holomorphic}. Ghiggini and Ni proved that a knot $K$ with Seifert genus $g$ is fibered if and only if $\dim \widehat{\hfk}_*(K, g) = 1$ \cite{ni2007knot, ghiggini2008knot}.

From these results, the following proposition follows. It uses the formalism of grid homology; precise definitions are recalled in the \emph{Notation and Prerequisites} section below.

\begin{proposition*}[I]
    \phantomsection\label{prop: converse}
	If there exists a grid diagram $\G$ representing the knot $K$ such that there exists a unique grid state $x \in S(\G)$ whose Alexander grading is maximal, then $K$ is fibered and has Seifert genus $A(x)$.
\end{proposition*}

\begin{proof}
    We use the representation of the simply blocked grid chain complex $\sbc$ from \cite[Rem. 4.6.13]{ozsvath2015grid}. Let $n$ be the grid size of $\G$ and let $A(x) = s$. Recall the action of the differential $\sbdifferential : \sbc \to \sbc$ on $x$
    \begin{gather*}
        \sbdifferential(x) = \sum_{y \in S(\G)} \sum_{\{r \in \eRect : r \cap \X = \emptyset, O_n(r) = 0\}} V_1^{O_1(r)} \cdots V_{n-1}^{O_{n-1}(r)} \cdot y.
    \end{gather*}
    The differential preserves the Alexander grading, a summand $V_1^{O_1(r)} \cdots V_{n-1}^{O_{n-1}(r)} \cdot y$ of $\sbdifferential(x)$ thus satisfies
    \begin{gather*}
        A(x) = A(V_1^{O_1(r)} \cdots V_{n-1}^{O_{n-1}(r)} \cdot y) = A(y) - \texttt{\#}(r \cap \Oh).
    \end{gather*}
    This implies $A(y) \geq A(x)$, by assumption, no grid state other than $x$ itself satisfies this condition. Since there can not exist a rectangle connecting $x$ with itself, it follows that $\sbdifferential(x) = 0$. The basis of $\widehat{\gc}_*(\G, s) = \bigoplus_d \widehat{\gc}_d(\G, s)$ are elements of the form $V_1^{k_1} \cdots V_{n-1}^{k_{n-1}} \cdot y$ such that $A(V_1^{k_1} \cdots V_{n-1}^{k_{n-1}} \cdot y) = s$. We conclude that $\widehat{\gc}_*(\G, s)$ is generated by the single element $x$ and that the differential $\sbdifferential$ is the zero map when restricted to $\widehat{\gc}_*(\G, s)$. For $\widehat{\gh}_*(K, s) = \bigoplus_d \widehat{\gh}_d(K, s)$ it follows
    \begin{gather*}
        \widehat{\gh}_*(K, s) = \frac{\ker \sbdifferential \cap \widehat{\gc}_*(\G, s)}{\im \sbdifferential \cap \widehat{\gc}_*(\G, s)} = \widehat{\gc}_*(\G, s).
    \end{gather*}
    In particular, $\dim(\widehat{\gh}_*(K, s)) = 1$ and $\widehat{\gh}_*(K, s') = 0$ for all $s' > s$. Due to the isomorphism of simply blocked grid homology and knot Floer homology $\widehat{\gh}_*(K, s) \cong \widehat{\hfk}_*(K, s)$ we conclude from the previous theorems that $K$ is fibered and $s$ is the Seifert genus of $K$.
\end{proof}

For brevity, we call such grid states \emph{maximal and unique}. It is natural to be interested in the converse statement of Proposition \nameref{prop: converse}.

\begin{question*}[i]
\phantomsection\label{qstn: central question}
    Does there exist a grid diagram with a maximal and unique grid state for every fibered knot?
\end{question*}

To analyze this question, we first establish an upper bound on the Alexander function in Section 1. Afterwards, in Section 2, we formulate an efficient method to check if a given grid diagram admits a grid state that realizes the upper bound for the Alexander function and is unique in that regard. Using these ideas, we create a simple Python program that searches for suitable grid diagrams of fibered knots in Section 3. We conclude the article with some remarks and a list of all the remaining unsolved fibered prime knots of crossing number $\leq 13$ in Section 4.

For a grid diagram of size $n$, there exist $n!$ different grid states. Using a brute-force approach to find a specific grid state thus requires calculating and comparing $n!$ gradings. The method we develop here requires at most $2n$ minimum detections on a matrix to decide if a grid diagram of size $n$ admits a grid state that is not only maximal and unique, but also realizes the upper bound of the Alexander function. The method also returns the unique grid state if it exists.

\addsec{Notation and Prerequisites}

This text addresses open problem 17.1.5 from \cite[Chapter 17]{ozsvath2015grid}. In particular, we work with the construction of simply blocked grid homology $\sbh$, as defined in \cite[Def. 4.6.12]{ozsvath2015grid}, the most important notions needed for the construction are listed below.

\begin{itemize}
    \item Planar grid diagrams \cite[Def. 3.1.1]{ozsvath2015grid} as well as toroidal grid diagrams \cite[Section 3.2]{ozsvath2015grid}. We denote toroidal grid diagrams by $\G$ and simply refer to them as grid diagrams. We call a planar representation of a toroidal grid diagram a planar realization.
    \item The set of grid states $S(\G)$ \cite[Def. 4.1.1]{ozsvath2015grid} of a grid diagram $\G$. Grid states are used to define the generators of the chain complex $\sbc$ of simply blocked grid homology.
    \item The set of rectangles $\Rect$ \cite[4.2]{ozsvath2015grid} between elements $x,y \in S(\G)$. Rectangles are needed to define the action of the differential $\sbdifferential$ of $\sbc$, cf. \cite[Rem. 4.6.13]{ozsvath2015grid}.
    \item The Alexander and Maslov functions $A: S(\G) \to \Z$ and $M: S(\G) \to \Z$ that induce the bigraded structure of $\sbh = \bigoplus_{d,s \in \Z} \sbhds{d}{s}$, cf. \cite[Section 4.3]{ozsvath2015grid}.
\end{itemize}

Throughout the text, all knots are assumed to be defined in the three-sphere $S^3$. 

\section{Alexander Function Upper Bound}

To define the upper bound of the Alexander function, we review its definition in terms of winding numbers, cf. \cite[Section 4.7]{ozsvath2015grid}.

\begin{definition}\label{def: winding number}
    Let $\G$ be a grid diagram of the knot $K$, let $p$ be a point on $\G$ that does not lie on the knot projection. The \emph{winding number} $\mathscr{w}(p)$ is an integer that describes how often the knot projection travels around $p$ counterclockwise. After choosing a planar realization of $\G$, the winding number at $p$ can be calculated by drawing a ray $\gamma$ emerging out of $p$ and then counting the number of intersections of $\gamma$ with the knot projection. An intersection where the knot projection travels counterclockwise increases $\mathscr{w}(p)$ by 1, clockwise intersections decrease $\mathscr{w}(p)$ by 1, cf. Figure \ref{fig: windingnumber}.
\end{definition}

\begin{figure}[tbp]
    \centering
    \includegraphics[width=12cm, height=5cm]{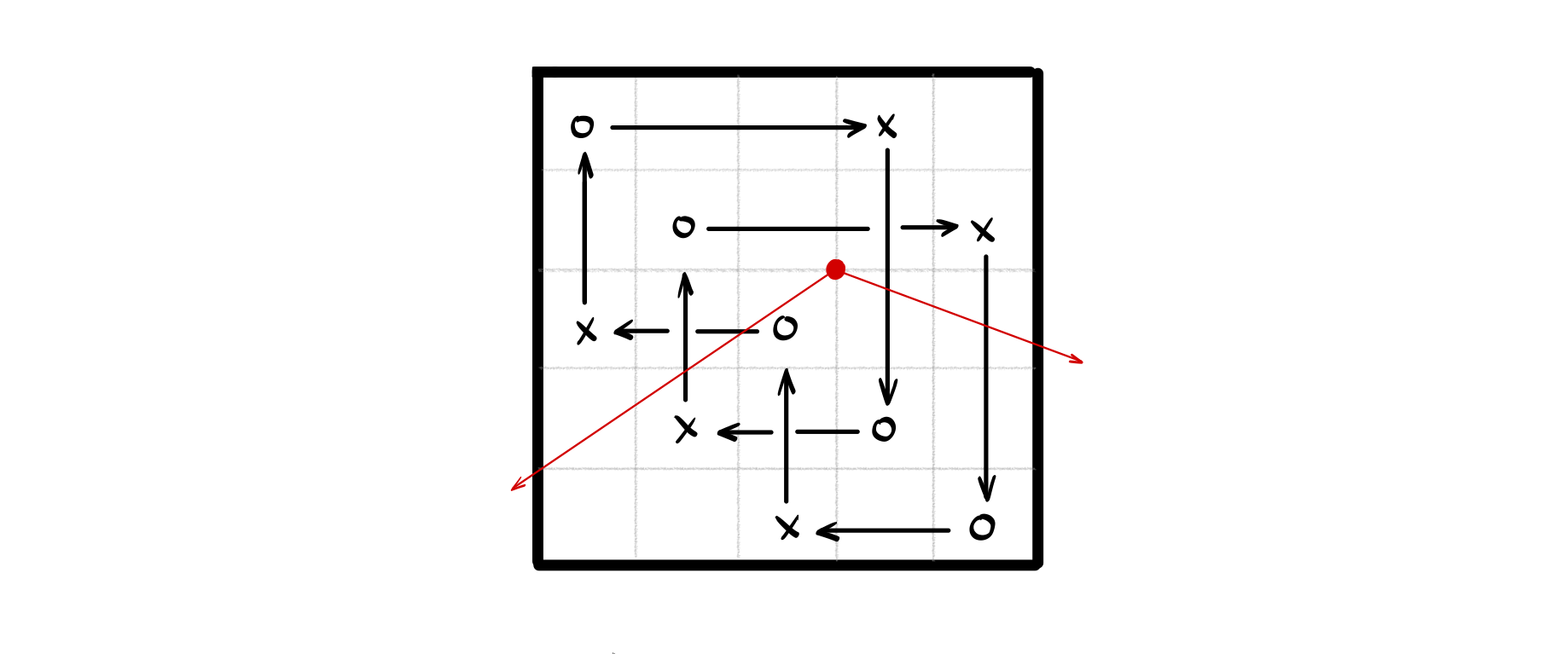}
    \caption{For this planar realization of a grid diagram, the winding number at the red point $p$ is -2. To see this, one can choose any ray emerging out of $p$ and count its signed intersection with the knot projection, as explained in Definition \ref{def: winding number}. Here, two possible rays are drawn, both intersect the knot projection twice in a clockwise direction.}
    \label{fig: windingnumber}
\end{figure}

\begin{definition}\label{def: alex grading alt def}
    Let $\G$ be a grid diagram of size $n$, then there are $2n$ small squares on the grid that are marked with either an $X$ or an $O$. Counting all four cornerpoints of all $2n$ squares gives a tuple $(p_1, \dots, p_{8n})$ of $8n$ lattice points on the grid, where it is possible that some points are counted more than once. The \emph{Alexander function} on a grid state $x = \{x_1, \dots, x_n\}$ is given by the equation
    \begin{gather*}\label{eq: Alexander}
        A(x) = -\sum_i \mathscr{w}(x_i) + \frac{1}{8} \sum_{j=1}^{8n} \mathscr{w}(p_j) - \left(\frac{n-1}{2}\right).
    \end{gather*}
\end{definition}

\begin{definition}
    Let $\G$ be a grid diagram of size $n$ and let $x = \{x_1, \dots, x_n\} \in S(\G)$. We define the \emph{winding function} $A' \colon S(\G) \to \Z$ as
    \begin{gather*}
        A'(x) = -\sum_{i = 1}^n \mathscr{w}(x_i).
    \end{gather*}
\end{definition}

$A(x)$ differs from $A'(x)$ only by the term $\frac{1}{8} \sum_{j=1}^{8n} \mathscr{w}(p_j) - (n-1)/2$ which is independent of $x$.

\begin{definition}
    The \emph{winding matrix} of a planar realization of a grid diagram $\G$ of size $n$ is an $n \times n$ matrix, whose entries are the winding numbers at the lattice points of $\G$, cf. Figure \ref{fig: grid state}.
\end{definition}

\begin{figure}[tbp]
    \centering
    \includegraphics[width=12cm, height=5cm]{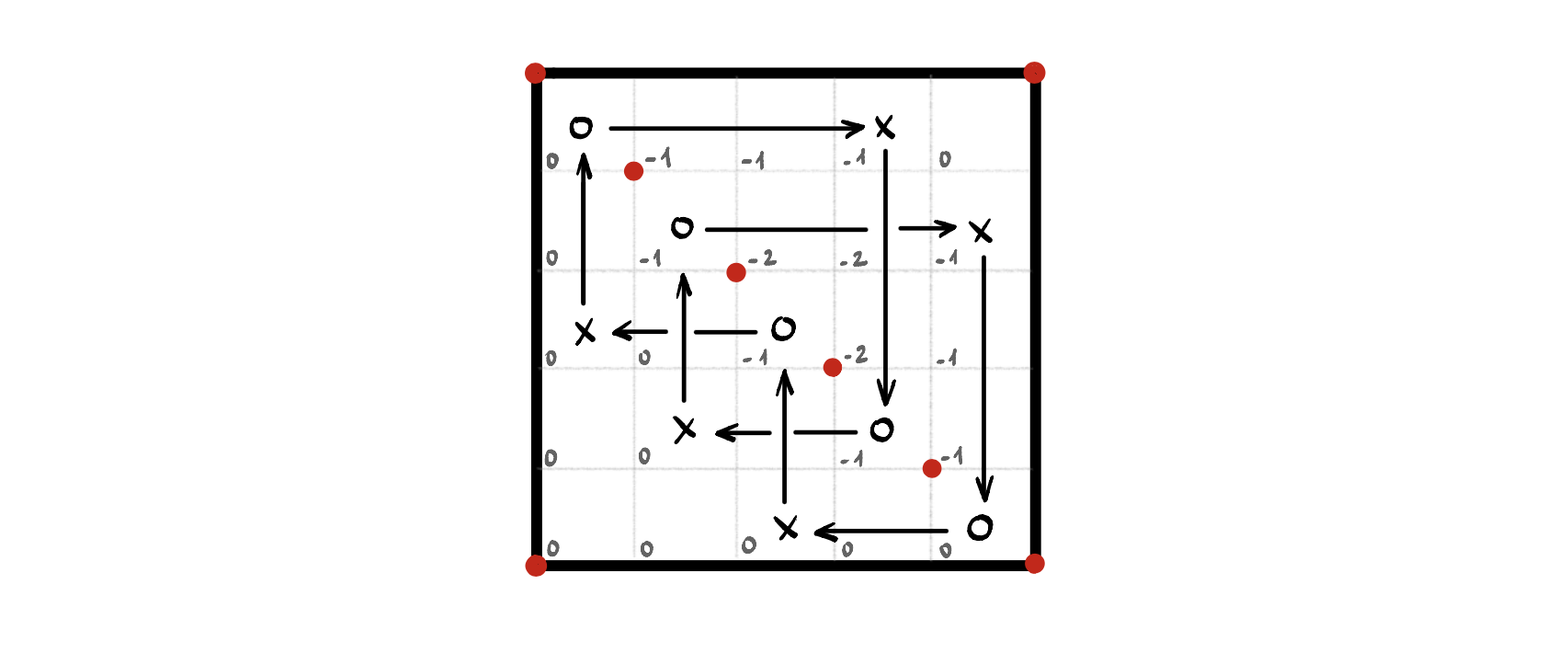}
    \caption{A planar realization of a grid diagram of the trefoil with a grid state visualized by red dots. Entries of the winding matrix are written in grey.}
    \label{fig: grid state}
\end{figure}

If we fix a planar realization of a grid diagram, we can label the vertical grid lines from left to right, and the horizontal grid lines from top to bottom. A grid state $x$ can be specified by a permutation $\sigma_x$ such that the point of $x$ on the $i$-th vertical line lies at the intersection with the $\sigma_x(i)$-th horizontal line. Using this notation, we see that the winding number of the point of $x$ in the $i$-th vertical line of the planar realization is therefore the $\sigma_x(i)$-th entry of the $i$-th column of the corresponding winding matrix, i.e., every row and every column of the winding matrix contains exactly one winding number of $x$.

\begin{definition}
    Given a grid diagram $\G$, let $M$ be the winding matrix of some planar realization of $\G$. Denote by $r_i$ the minimum of the $i$-th row of $M$, by $c_i$ the minimum of the $i$-th column of $M$. We call $\row = -\sum_i r_i$ the \emph{row number} of $\G$ and $\col = -\sum_i c_i$, the \emph{column number} of $\G$.
\end{definition}

As any two winding matrices of the same grid diagram only differ by a sequence of cyclic permutations of the rows and columns, it follows that $\row$ and $\col$ are independent of the chosen planar realization. The row- and column-numbers therefore give rise to the following upper bounds.

\begin{gather*}
    A'(x) \leq \min\{\row, \col\}\\
    A(x) \leq \min\{\row, \col\} + \frac{1}{8} \sum_{j=1}^{8n} \mathscr{w}(p_j) - (n-1)/2.
\end{gather*}

\begin{definition}
    We call a grid state $x$ \emph{perfect} if it realizes the upper bound for the Alexander function, i.e. if $A'(x) = \min\{\row, \col\}$. We call $x$ \emph{row-perfect} if $A'(x) = \row$ and \emph{column-perfect} if $A'(x) = \col$.
\end{definition}

Row- and column-numbers enable us to quickly check if certain grid states are maximal. Furthermore, in Proposition \ref{prop: idea} of the next section, we show that there exists an efficient method of analyzing if a grid diagram has a unique perfect grid state.

\begin{example} \label{ex: unique perfect grid state}
    The grid state in Figure \ref{fig: grid state} is the unique perfect grid state of this diagram. Indeed, for the winding matrix in the figure, we have $\row = \col = 6$. The grid state $x$, given by the purple dots, realizes this upper bound since $A'(x) = 6$. To see the uniqueness, note that any maximal grid state of this grid diagram must have its corresponding entry in the third row of the winding matrix $M$ positioned at the same spot, as this is the only point where the entry is minimal for this row. This forces the corresponding entry in the second and fourth rows to also be at the same spot as in the given grid state. Continuing like this, the entries in the first and, finally, the last row of $M$ are also fixed.
\end{example}

\section{Finding Unique Perfect Grid States}

Here, we formulate an efficient method of analyzing whether a grid diagram has a unique perfect grid state. We start with a technical definition.

\begin{definition}
    Let $\G$ be a grid diagram of size $n$, let $M$ be the winding matrix of a fixed planar realization of $\G$ with columns $c_1, \dots, c_n$. Let $x$ be a grid state described by the permutation $\sigma$. We say that $(M, \sigma)$ \emph{admits a loop} if there exist columns $c_{i_1} \dots c_{i_k}$ of $M$ and for each column an index $j_{i_1} \dots j_{i_k}$ such that $j_{i_r} \neq \sigma(i_r)$ and $M_{j_{i_r}, i_r} = M_{\sigma(i_r), i_r}$ for all $1 \leq r \leq k$. Additionally, we require that the sequence of matrix indices
    \begin{gather*}
        (\sigma(i_1), i_1), (j_{i_1}, i_1), (\sigma(i_2), i_2), (j_{i_2}, i_2), \dots, (\sigma(i_k), i_k), (j_{i_k}, i_k)
    \end{gather*}
    describes a loop in $M$, i.e. $j_{i_r} = \sigma(i_{r+1})$ for all $1 \leq r \leq k-1$ and $j_{i_k} = \sigma(i_1)$.
\end{definition}

\begin{example}
    The following $5 \times 5$ matrix $M$ admits a loop for the permutation $\sigma = (1,3,4,2,5)$. The first three columns contain suitable indices $j_1, j_2, j_3$. For $1 \leq i \leq 3$, the entries $M_{\sigma(i), i}$ are red and $M_{j_i, i}$ are blue.
    \begin{gather*}
        \begin{pmatrix}
            \color{red}{-1}  & 0                & \color{blue}{-1} & 2 & 0  \\
            0                & 1                & 0                & 0 & 0  \\
            \color{blue}{-1} & \color{red}{-2}  & 1                & 1 & -2 \\
            -1               & \color{blue}{-2} & \color{red}{-1}  & 3 & -2 \\
            1                & -1               & 3                & 1 & -3
        \end{pmatrix}
    \end{gather*}
\end{example}

The notion of a loop is useful, as the following lemma shows.

\begin{lemma}\label{lem: loop}
    Let $\G$ be a grid diagram of size $n$ and let $M$ be the winding matrix of a fixed planar realization. Let $x \in S(\G)$ be described by the permutation $\sigma$. If $(M, \sigma)$ admits a loop, then there exists another grid state $y \neq x$ such that $A(y) = A(x)$.
\end{lemma}

\begin{proof}
    Let $c_{i_1}, \dots, c_{i_k}$ be the columns of $M$ with indices $j_{i_1}, \dots, j_{i_k}$ that create the loop for $(M, \sigma)$. To simplify the notation, assume $i_1 = 1, \dots, i_k = k$. We can then define the sequence
    \begin{gather*}
        \tau = (j_1, \dots, j_k, \sigma(k + 1), \dots, \sigma(n)).
    \end{gather*}
    By construction, $\tau$ is a permutation that is different from $\sigma$. Indeed the looping condition implies that $j_1, \dots, j_k$ are all different and $j_r \neq \sigma(r')$ for $1 \leq r \leq k$ and $r' \in \{k+1, \dots, n\}$. Let $y$ be the grid state of $\G$ with grid state permutation $\tau$. The claim $A(y) = A(x)$ then follows from the assumption $M_{j_r, r} = M_{\sigma(r), r}$.
\end{proof}

We can now formulate and prove an important proposition about the existence of unique perfect grid states.

\begin{proposition}\label{prop: idea}
    Let $\G$ be a grid diagram that has a unique column-perfect grid state. Let $M$ be the winding matrix of a fixed planar realization of $\G$. Then there exists a column of $M$ with a unique minimal entry. The analogous statement holds for row-perfect grid states.
\end{proposition}

\begin{proof}
    We formulate the proof for column-perfect grid states; the row-perfect case works similarly. Let $n$ be the grid size of $\G$ and assume contrapositively that all columns of $M$ have more than one entry that is minimal for this column. If the diagram does not admit a column-perfect grid state, we are done, otherwise there exists a permutation $\sigma$ that represents a column-perfect grid state $x$. Additionally, for all columns $1 \leq i \leq n$ of $M$ there exists an index $j_i \neq \sigma(i)$ such that $M_{j_i,i} = M_{\sigma(i),i}$. If we manage to create a permutation $\tau$ different from $\sigma$ such that at each index $i$ we have $\tau(i) \in \{\sigma(i), j_i\}$, the claim follows. This can be achieved using Algorithm \ref{alg: new perm}.

    The Algorithm can be visualized as follows. We start by connecting the point of $x$ that corresponds to $\sigma(1)$ by a vertical line to the point on the grid diagram that corresponds to $j_1$. We then connect this point to the only point of $x$ on the same horizontal line, i.e., to the point represented by $\sigma(t)$ for some $1 \leq t \leq n, t \neq 1$, and again connect this point by a vertical line to the point corresponding to $j_t$. Continuing like this, two things can happen. If the piecewise linear curve meets a point of $x$ twice, a loop is created which specifies a new column-perfect grid state as explained in Lemma \ref{lem: loop}. Otherwise, the modified version of $\sigma$ is at some point again a permutation that is different from the original permutation, cf. Figure \ref{fig: three images}.
\end{proof}

\begin{figure}[tbp]
    \centering
    \includegraphics[width=12cm, height=5cm]{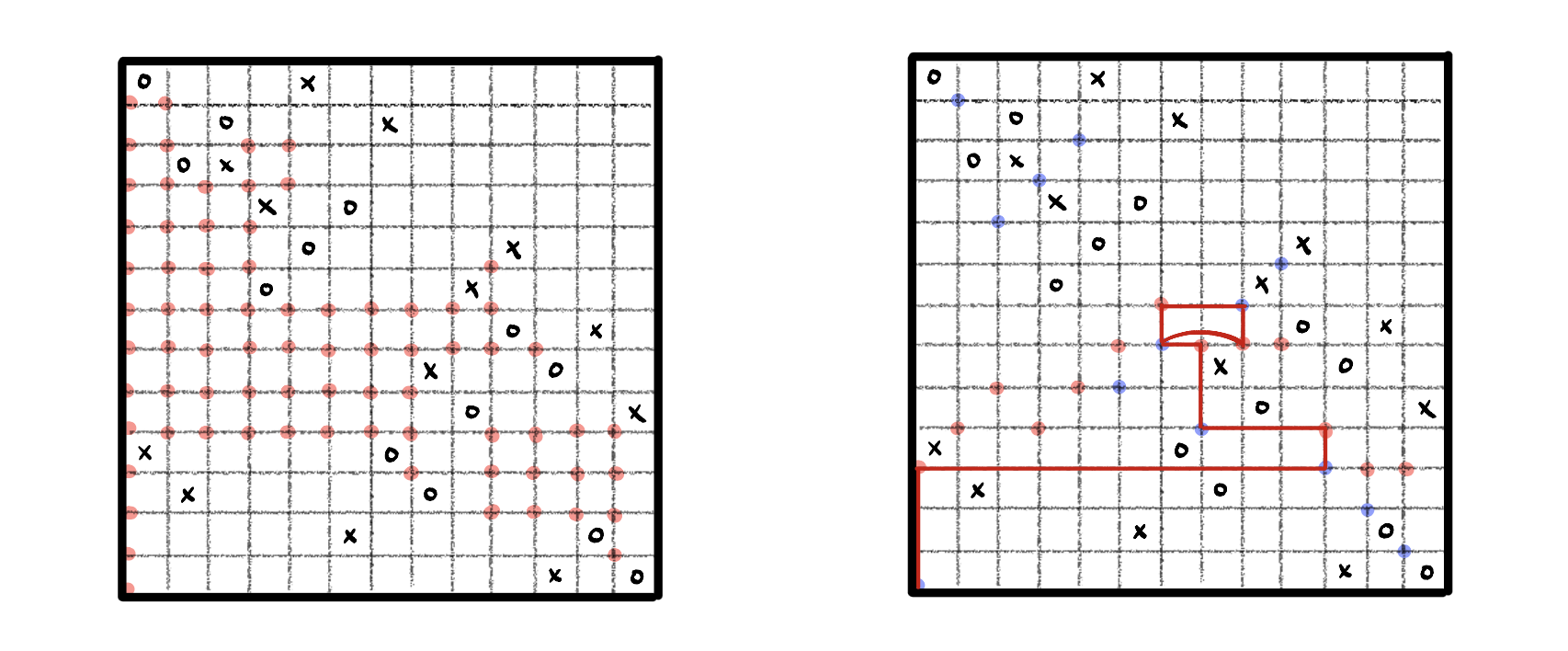}
    \caption{
        A visual explanation of Proposition \ref{prop: idea}.\\
        \textit{Left}: An example of a grid diagram where no column of the winding matrix has a unique minimal value. The red dots represent entries of the winding matrix that are minimal for this column.\\
        \textit{Right}: For each column, we choose two minimal entries (blue and red). The blue dots represent a column perfect grid state. Then we draw a curve (red) until a point of the grid state is visited twice. The closed loop, which is created from this curve, tells us which blue dots can be switched to red dots to create a new column perfect grid state.
    }
    \label{fig: three images}
\end{figure}

\begin{algorithm}
    \caption{New Perfect Grid State}\label{alg: new perm}
    \begin{algorithmic}[1]
        \Require A permutation $\sigma$ of length $n$ and a list $J$ of the same length,
        such that $J[i] \neq \sigma[i]$ for all $1 \leq i \leq n$.
        \Ensure A new permutation $\tau$ different from $\sigma$, such that at each index $i$ the new permutation has the value $\sigma(i)$ or $J(i)$.
        \State $n \gets \text{length}(\sigma)$
        \State $S \gets \text{copy}(\sigma)$
        \State $S[1] \gets J[1]$ \Comment{{\footnotesize Replace first element with value from $J$}}
        \State $i \gets 1$ \Comment{{\footnotesize Index of most recent change in $S$}}
        \State $changes \gets [\, i \,]$ \Comment{{\footnotesize List of modified indices.}}
        \While{true}
        \State $duplicates \gets \text{duplicate elements of S}$
        \If{$duplicates = \emptyset$}
        \Comment{{\footnotesize No duplicate found: the permutation is valid}}
        \State $\tau \gets \text{copy}(S)$
        \State \Return $\tau$
        \Else
        \State Let $t$ be the unique element in $duplicates$ with $t \neq i$
        \EndIf
        \If{$t \in changes$}
        \Comment{{\footnotesize Loop detected: modify only indices between $t$}}
        \State $start \gets \text{position of } t \text{ in } changes$
        \State $loop \gets \text{elements of } changes \text{ from position } start+1 \text{ onward}$
        \State $\tau \gets \text{copy}(\sigma)$
        \ForAll{$k \in loop$}
        \State $\tau[k] \gets J[k]$
        \EndFor
        \State \Return $\tau$
        \EndIf
        \State $S[t] \gets J[t]$ \Comment{{\footnotesize Update $S$ at t with the value from $J$}}
        \State $i \gets t$ \Comment{{\footnotesize Update current index to recently changed index}}
        \State Append $i$ to $changes$
        \EndWhile
    \end{algorithmic}
\end{algorithm}

The previous proposition can be extended to an efficient method for determining whether a grid diagram $\G$ admits a unique perfect grid state. Fix a planar realization of $\G$ with winding matrix $M$ and compute $\row$ and $\col$. If one of the values is strictly bigger, there can only exist one type of perfect grid state, say column-perfect grid states. We then check whether there exists a column of $M$ with a unique minimal entry. If not, we immediately know from Proposition \ref{prop: idea} that $\G$ does not have a column-perfect grid state and thus no perfect grid state at all. If such a column exists, we save the unique minimal entry, then delete the column and row of $M$ containing that entry. For the new reduced matrix, three things can happen:

\begin{enumerate}
    \item[-] There are columns that now no longer have their original minimum as an entry. In that case, we know that $\G$ does not admit a column-perfect grid state.
    \item[-] For all remaining rows, there does not exist a column that only has one entry with minimal value. If this happens, we know that $\G$ does not admit a unique column-perfect grid state by the same argument as in Proposition \ref{prop: idea}. It could be that $\G$ admits several perfect grid states or none.
    \item[-] There again exists at least one column that has a unique entry with the original minimal value of the column. Here, we can continue as before by extracting the original position of the new unique minimal entry and then reducing the matrix again.
\end{enumerate}

This process either stops after finding a unique column-perfect grid state, or it stops if, after some iteration, case 1 or 2 happens. In the case that $\row = \col$, we do have to perform this method for columns and rows. A downside of this procedure is that it cannot tell grid diagrams that admit more than one column-perfect grid state apart from those that admit no column-perfect grid state. This would be interesting to know, since if $\G$ admits no column-perfect grid state, it may still have a maximal and unique grid state. An example of this is given in Figure \ref{fig: max unique, not perfect}.

\begin{example}
    The grid state $x$, shown by the red dots in Figure \ref{fig: max unique, not perfect}, is maximal and unique but not perfect. Indeed, it is not perfect since $A'(x) = 13$ whereas $\row = 14 = \col$. Note that the 5th and 6th rows of the corresponding winding matrix have their unique minimal entry at the same column; it is thus impossible to have a row-perfect grid state. Analogously, the 3rd and 4th column have their unique minimal entry in the same row. Therefore, column-perfect grid states also cannot exist, and we see that $x$ is maximal. Any other maximal grid state must have its corresponding values on the winding matrix of the 3rd or 4th column on the 5th row; this forces the entries of the 2nd, 5th, 6th, 7th, and 8th columns to be in the same spot as for $x$. It is now easy to see that $x$ is a unique maximal grid state.
\end{example}

\begin{figure}[tbp]
    \centering
    \includegraphics[width=12cm, height=6cm]{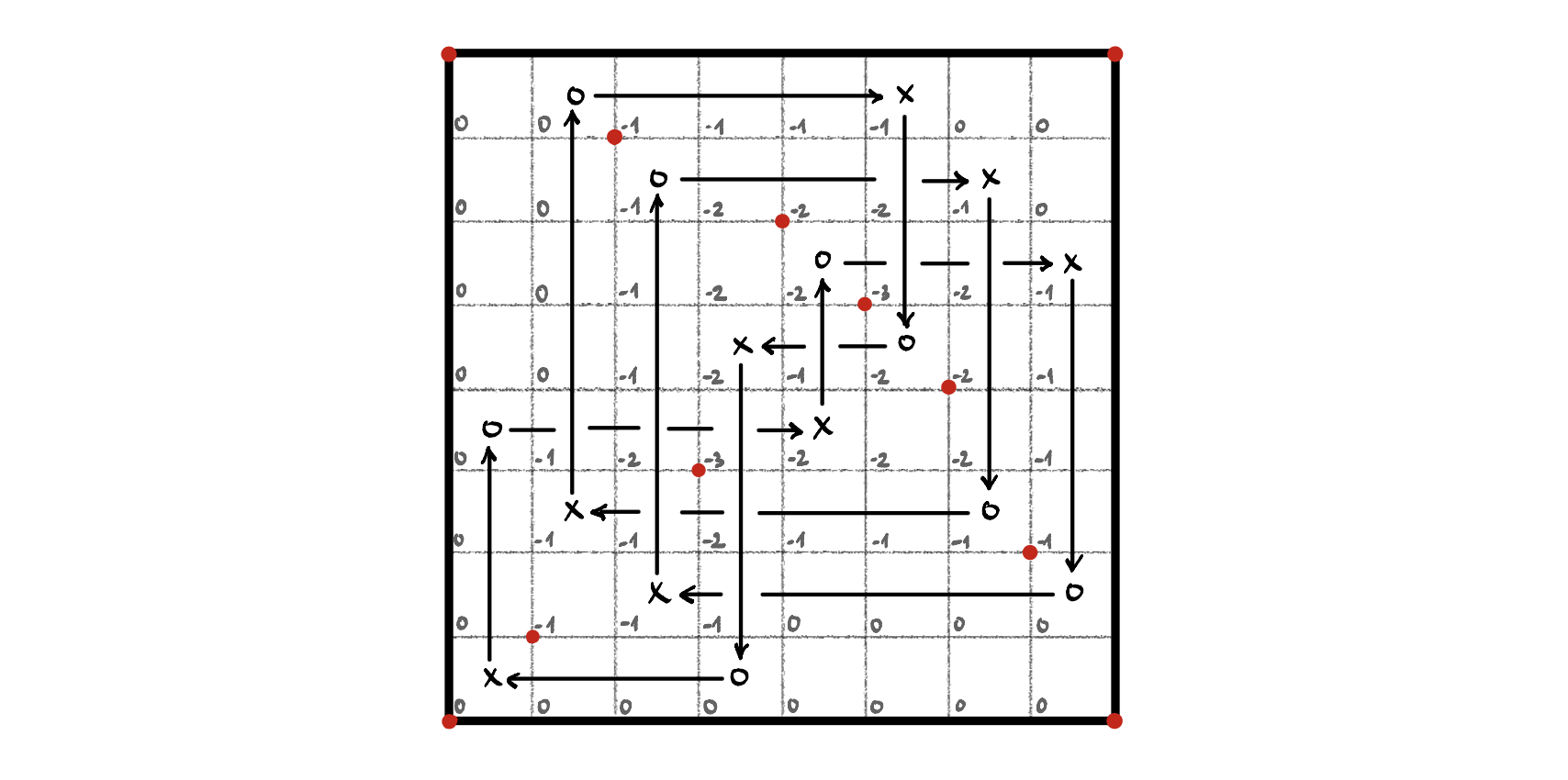}
    \caption{An example of a grid diagram with a grid state that is maximal and unique but not perfect.}
    \label{fig: max unique, not perfect}
\end{figure}

While our method for proving the uniqueness of maximal grid states is fast when the grid state is perfect, we can see from the previous example that it is much more tedious to show uniqueness when the grid state is not perfect. In Figure \ref{fig: max unique, not perfect}, the grid state $x$ is almost perfect in the sense that $A'(x) + 1 = \min\{\row, \col\}$. In general, the bigger the difference between $\min\{\row, \col\}$ and $A'(x)$, the less practical the method becomes. We thus focus on grid diagrams of knots that admit a unique perfect grid state. For convenience, we introduce the following definition.

\begin{definition}
    Given a grid diagram $\G$ of a knot $K$, we call $\G$ \emph{nice}, if there exists a unique perfect grid state for $\G$.
\end{definition}

Instead of working on Question \nameref{qstn: central question}, we reformulate it into:

\begin{question}\label{qstn: updated central question}
    Does there exist a nice grid diagram for every fibered knot?
\end{question}

\section{A Python programm to Find Nice Grid Diagrams}

This last chapter explains the Python program \emph{griddiagrams} \cite{griddiagrams} that was developed based on the ideas of the previous chapters. Its purpose is to find nice grid diagrams of fibered knots by creating a large list of candidate diagrams that represent the input knot. Here we explain the key functions; further documentation can be found in the source code.

Recall from \cite[Section 3.1.2]{ozsvath2015grid} that grid moves, i.e., commutations and (de)-stabilizations, are modifications of grid diagrams that do not change the underlying knot type. They can be used to create a large list of diagrams of the same knot. For each grid diagram, we want to apply the efficient method from the previous section to find a unique perfect grid state. While this approach is very brute force, most of the required calculations can be implemented very efficiently.

We use the knot database provided from \emph{https://knotinfo.org} \cite{knotinfo}. It contains 5397 fibered prime knots and, conveniently, all of the knots in this database are already equipped with a minimal grid diagram representation. A grid diagram is given in so-called \emph{grid-notation}, which is a list of tuples that specify the positions of the $X$- and $O$-markings. For example, the diagram of size 5 from Figure \ref{fig: grid gives link} would be given by the following sequence of 10 tuples
\begin{gather*}
    (1,1), (1,3), (2,2), (2,4), (3,3), (3,5), (4,1), (4,4), (5,2), (5,5).
\end{gather*}

\begin{figure}[tbp]
    \centering
    \includegraphics[width=12cm, height=5cm]{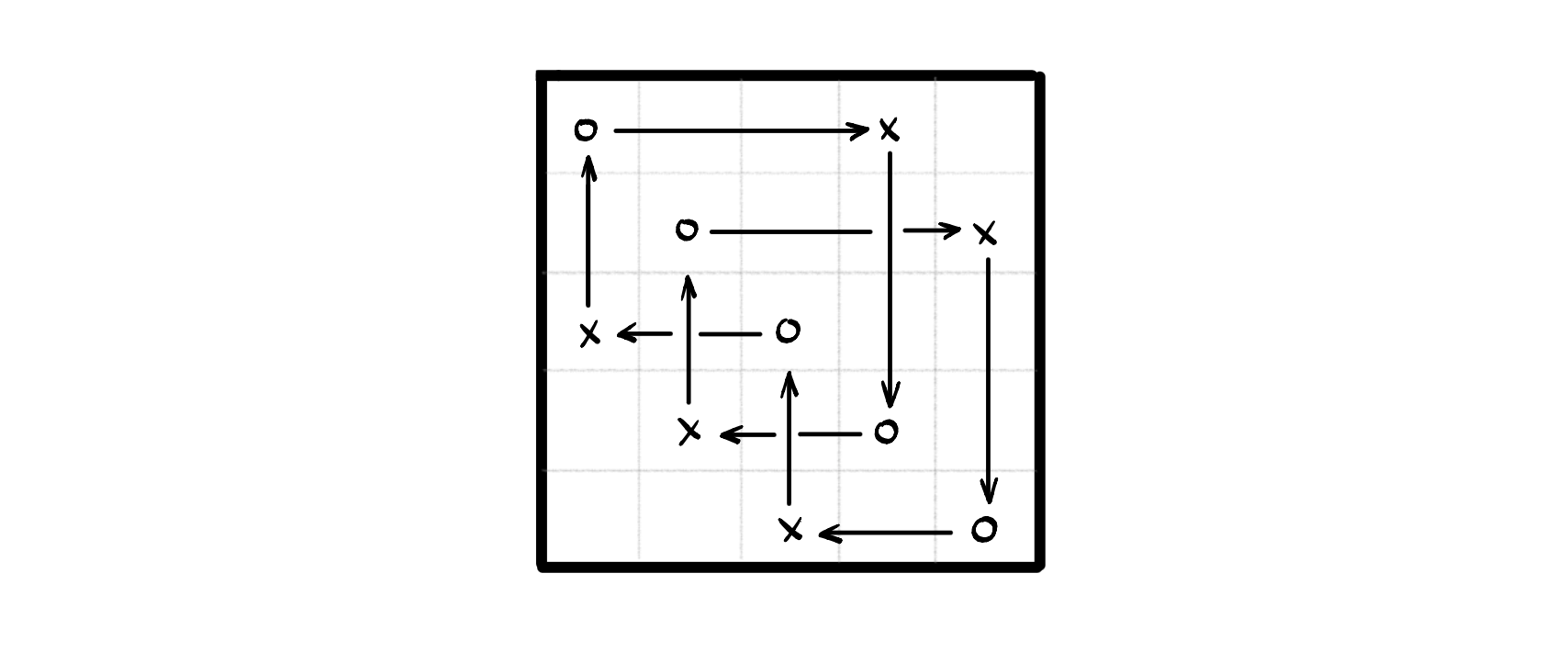}
    \caption{A planar grid diagram of the trefoil knot.}
    \label{fig: grid gives link}
\end{figure}

In this notation, the orientation of the grid diagram is not specified. In order to perform column-commutations efficiently, it is much more practical to describe grid diagrams by vertical segments. In this description, Figure \ref{fig: grid gives link} is specified by 5 oriented vertical intervals
\begin{gather*}
    (3,1), (4,2), (5,3), (1,4), (2,5).
\end{gather*}
We call this the \emph{vertlist-notation}. The Python program transforms grid-notation into vertlist-notation in two steps, by first transforming grid-notation into an intermediate format with the function \lstinline[style=python]{gridnotation_to_gridlist()}, and then transforming this format into vertlist-notation using \lstinline[style=python]{vlist()}.

The orientation that is chosen during this transformation process does not matter. Reversing the placements of the $X$- and $O$-markings on a grid diagram amounts to reversing the orientation of the underlying knot. This is implemented in the function \lstinline[style=python]{rev()}, which is used later to look for unique perfect grid states in both orientations. In vertlist-notation, a simple function, \lstinline[style=python]{can_commute()}, is enough to look for possible column-commutations. To efficiently perform row-commutations, we describe the grid diagram using its vertical segments. We call this the \emph{horzlist-notation}. Another simple Python function, \lstinline[style=python]{v_to_h()}, transforms vertlist- into horzlist-notation, and we again use \lstinline[style=python]{can_commute()} to check if neighbouring rows can be commuted. The next step is to take a grid diagram given as a vertlist and compute a list of all grid diagrams obtainable after one row- or column commutation. This is done using \lstinline[style=python]{c_move()}. As we are thinking of the grid diagram embedded on the torus, this function allows column commutations between the first and last vertical segment as well as row commutations between the highest and lowest row. 

Another benefit of describing grid diagrams in vertlist-notation is that calculating the winding matrix is straightforward; cf. \lstinline[style=python]{w_matrix()}, which was inspired by the program Gridlink \cite{GridLink}. On a matrix of a knot, we can look for unique row- and column-perfect grid states using the method derived from Proposition \ref{prop: idea}. The function \lstinline[style=python]{h_type_0_permutation()}\footnote{The description "h-type-0" stands for grid states x whose difference $\mathscr{r}-A'(x)$ is zero.} finds unique row-perfect grid states. Similarly, we define \lstinline[style=python]{v_type_0_permutation()} for unique column-perfect grid states. According to Proposition \nameref{prop: converse}, the Alexander grading of a unique perfect grid state coincides with the Seifert genus of the knot. The program thus also includes \lstinline[style=python]{a_grading()}, a function that calculates the Alexander grading of a grid state to verify whether the unique perfect grid state the code finds satisfies this necessary condition. All of this is combined in \lstinline[style=python]{try_permutations()}, which is the function that tries to find a unique perfect grid state for a given diagram.

We now have all the building blocks to create a first breadth-first search (BFS) algorithm, \lstinline[style = python]{gridstate_finder_commute()}, that looks for a nice grid diagram by creating a large list of representations using only commutation moves. If the algorithm is successful, it results in a nice grid diagram of minimal size since commutation moves do not change the grid size. This approach successfully finds nice grid diagrams for 4346 of the 5397 fibered prime knots.

For the remaining knots, we use a similar approach after performing a stabilization. Any stabilization can be obtained by a stabilization of type X-SW followed by a sequence of commutations \cite[Cor. 3.2.3]{ozsvath2015grid}. Analogously, it is also enough to perform only stabilizations of type X-NW in combination with commutations. The function \lstinline[style = python]{x_nw()} performs an X-NW stabilization of a grid diagram in vertlist-notation on a specified segment. The second BFS-algorithm, \lstinline[style=python]{gridstate_finder_stab()}, then tries to find a nice diagram after performing one stabilization. This algorithm is noticeably slower than \lstinline[style = python]{gridstate_finder_commute()}, which is most likely caused by the fact that a stabilization introduces a segment of length one that can be commuted with everything, thus increasing the number of possible commutations drastically. Nevertheless, it successfully finds nice grid diagrams for 1039 of the remaining knots. Only 12 fibered prime knots with crossing number $\leq 13$ are left unsolved.

\section{Final Remarks and Questions}

The remaining fibered prime knots with crossing number $\leq 13$ for which no nice grid diagram has been found are listed in Table \ref{tab: unsolved_knots}.

\begin{table}[tbp]
    \renewcommand{\arraystretch}{1.3}
    \begin{tabularx}{\textwidth}{|>{\centering\arraybackslash}m{3cm}|>{\centering\arraybackslash}X|}
        \hline
        \textbf{Crossing number} & \textbf{Unsolved knots}                                              \\
        \hline
        $\leq 11$                & $\emptyset$                                                          \\
        \hline
        12                       & $12n_{79}$, $12n_{168}$                                              \\
        \hline
        \multirow{2}{*}{13}      & $13n_{282}$, $13n_{917}$, $13n_{1279}$, $13n_{1281}$, $13n_{1413}$,  \\
                                 & $13n_{1826}$, $13n_{2915}$, $13n_{3089}$, $13n_{3904}$, $13n_{3932}$ \\
        \hline
    \end{tabularx}
    \caption{Unsolved knots by crossing number}
    \label{tab: unsolved_knots}
\end{table}

Coincidentally, for all unsolved knots, the minimal grid size coincides with the crossing number. In the full dataset, 35\% of the knots have a minimal grid size of 14 or 15. For all of those knots, nice grid diagrams have been found, cf. Figure \ref{fig: nice grids} for two examples of nice grid diagrams of grid size 16, i.e., these are examples of knots of minimal grid size 15 where the code only found a nice grid diagram after stabilizing once.

\begin{figure}[tbp]
    \centering
    \includegraphics[width=6.5cm, height=6.5cm]{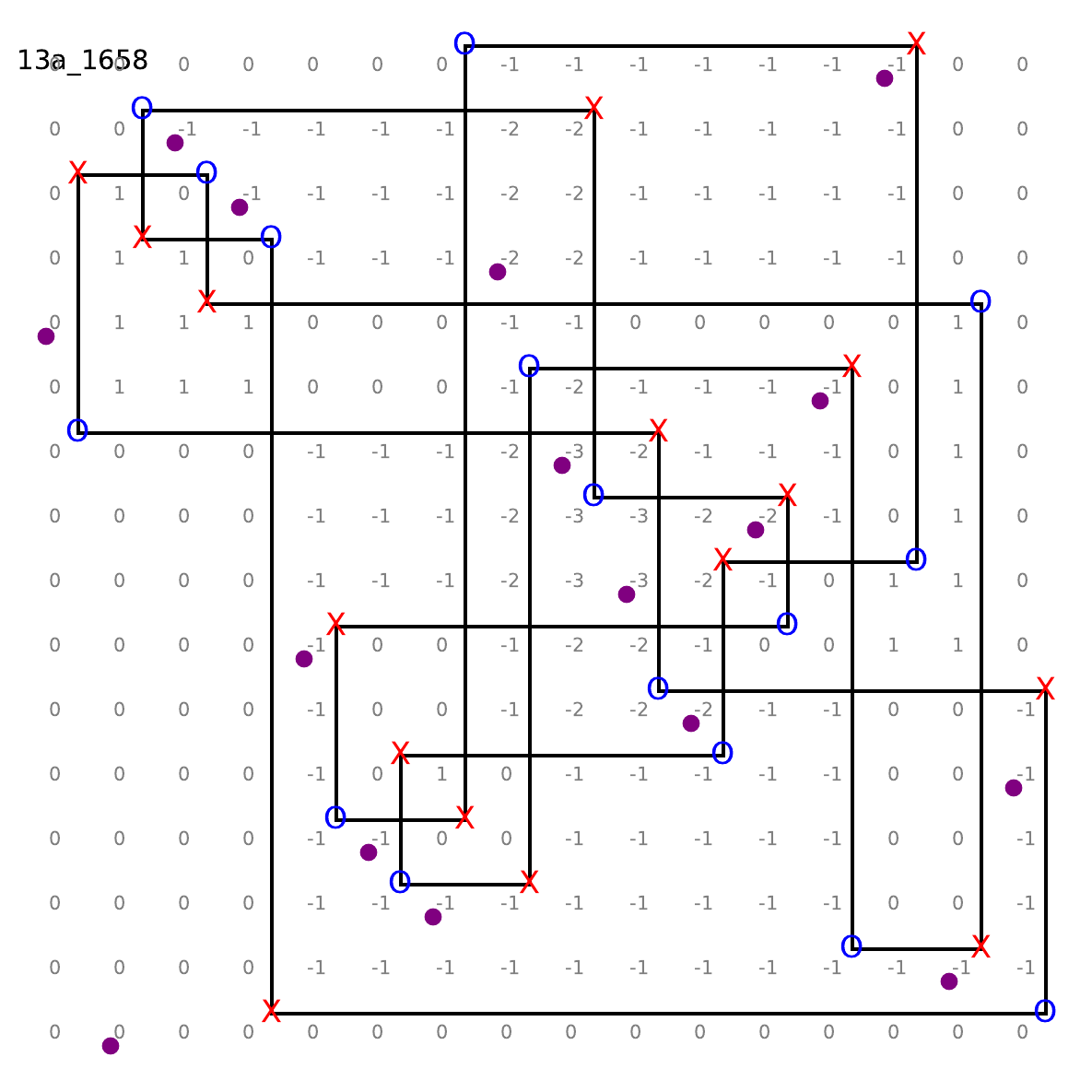}
    \includegraphics[width=6.5cm, height=6.5cm]{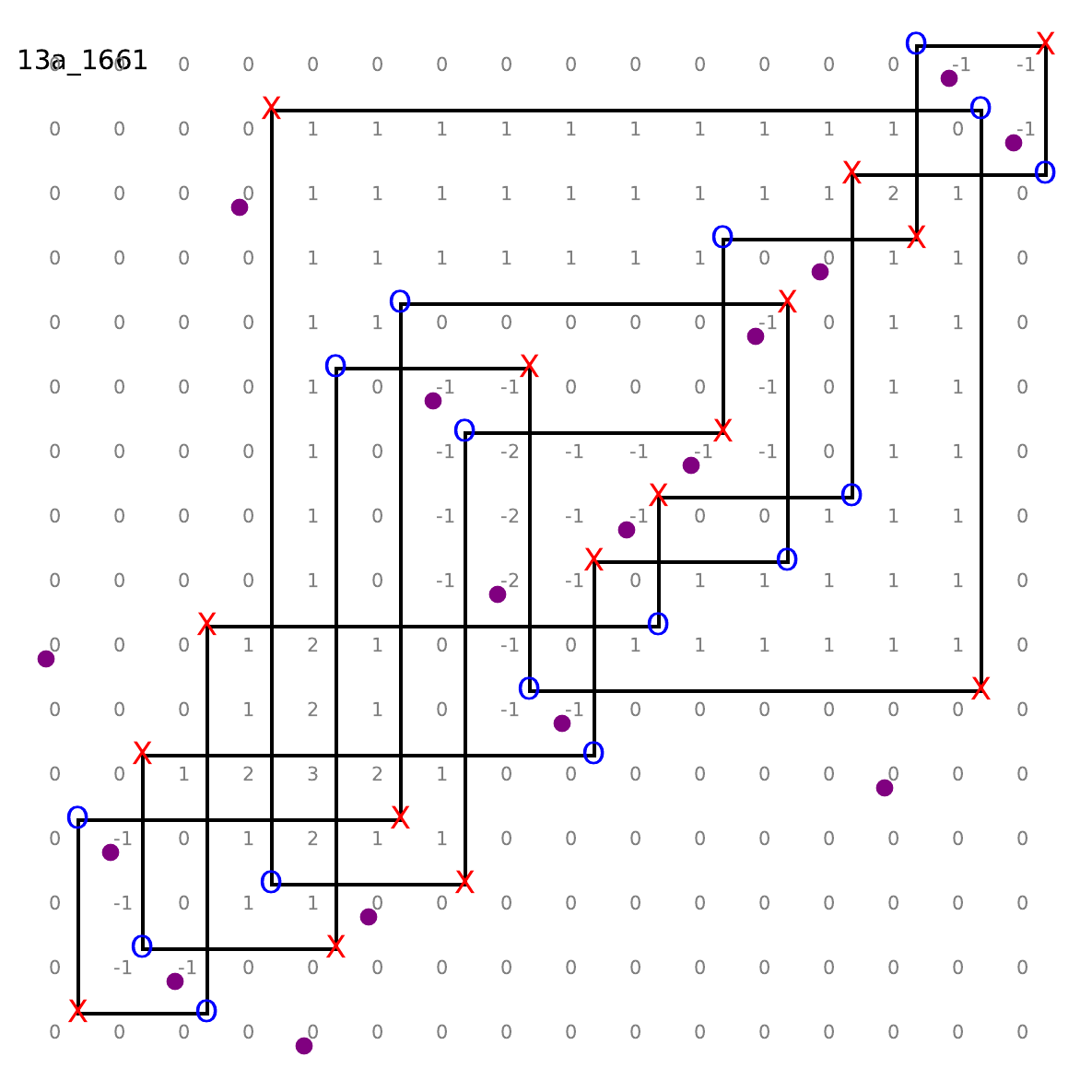}
    \caption{
        Two examples of nice grid diagrams for fibered prime knots of minimal grid size 15 where one stabilization has been performed. The purple dots represent the corresponding perfect grid state.
    }
    \label{fig: nice grids}
\end{figure}

For the unsolved knots, introducing a second stabilization did not yield any nice grid diagrams. However, to properly investigate grid diagrams where more than one stabilization has been performed, it would be necessary to make the algorithms much more efficient. One possibility is to use some form of multiprocessing. Within the project, multiprocessing has already been used to find nice grid diagrams for several knots simultaneously. Nevertheless, for the most time-consuming knots, a parallelized version of the BFS algorithms would be much faster. It remains an interesting task to create more efficient versions of our code. As of right now, the functionalities of the code are also limited; some of the missing features are:

\begin{itemize}
    \item[-] The code is only able to perform stabilizations of type $X-NW$. While this is sufficient in theory, introducing additional stabilization types may drastically reduce the number of necessary commutations.
    \item[-] The code does not destabilize diagrams.
    \item[-] We are not looking at 2 or more stabilizations.
\end{itemize}

Even without introducing destabilizations and using at most one stabilization on a minimal grid diagram, this simple Python program was able to find nice diagrams for 5384 of the 5397 fibered prime knots in the database. This suggests the following question.

\begin{question}
    Does every knot that admits a grid diagram with a unique maximal grid state also admit a nice grid diagram?
\end{question}

This project has led to many more interesting questions, for example:

\begin{itemize}
    \item[-] For those knots where nice grid diagrams were found only after stabilizing, do there exist nice grid diagrams of minimal grid size? We did not introduce destabilizations in the package. This may be related to the question of whether all grid diagrams of a knot of the same size are connected using commutations only, or if it is necessary to stabilize and destabilize. Dynnikov proved that any grid diagram of the unknot can be reduced to the simplest diagram without stabilizing \cite{dynnikov2006arc}. However, the documentation of the program Gridlink \cite{GridLink} points to the appendix of \cite{menasco2006addendum} for an example of a grid diagram of a knot that can not be reduced without stabilizing.
    \item[-] Assuming a knot has a minimal grid diagram with a unique maximal grid state as well as a nice grid diagram, does there exist a minimal nice grid diagram? So far, the code has not found a minimal nice grid diagram for $9_{42}$, but Figure \ref{fig: max unique, not perfect} is an example of a minimal grid diagram for $9_{42}$ with a unique maximal grid state.
\end{itemize}

We end this project with a list of remarks.

\begin{itemize}
    \item[-] Now that nice grid diagrams have been found for a large list of fibered prime knots, it might be interesting to use machine learning methods in order to find more efficient methods of finding nice grid diagrams. As of now, there is no measure of how far a grid diagram is from being nice; maybe machine learning methods can even help find such a metric.
    \item[-] The package checks if the Alexander grading of the unique perfect grid state coincides with the 3-genus of the input knot. According to Lemma \nameref{prop: converse}, this has to be the case. For all knots where a nice grid diagram has been found, the grading of the grid state coincided with the genus of the knot. This is a very strong sign that the code is producing correct results. Additionally, one can compare the homology ranks of the resulting grid diagram with the original knot using, for example, Gridlink \cite{GridLink}, to further ensure that the calculations are correct. This has been done on a sample basis for several grid diagrams, and the ranks have always been correct.
\end{itemize}

\newpage

\printbibliography

\end{document}